\documentclass[11pt]{article}
\usepackage{amsmath,amsfonts,amssymb,amsthm,fancyhdr,bm,mathtools,enumitem}
\usepackage[hidelinks]{hyperref}
\usepackage{fullpage}
\usepackage[dvipsnames]{xcolor}
\usepackage[capitalize]{cleveref}
\usepackage[color=Purple!50!white]{todonotes}
\usepackage[numbers, square,comma,sort&compress]{natbib}

\usepackage[affil-it]{authblk}
\newtheorem{thm}{Theorem}[section]
\newtheorem*{thm*}{Theorem}

\newtheorem{claim}[thm]{Claim}

\newtheorem{conj}[thm]{Conjecture}

\theoremstyle{definition}

\crefname{equation}{equation}{equations}
\crefname{lem}{Lemma}{Lemmas}
\crefname{thm}{Theorem}{Theorems}

\newlist{lemenum}{enumerate}{1}
\setlist[lemenum]{label=(\alph*), ref=\thelem(\alph*)}
\crefalias{lemenumi}{lemma}

\newcommand{\flo}[1]{\lfloor #1 \rfloor}

\newcommand{\rc}{{\hat{r}}_c}

\title{Connected size Ramsey numbers of matchings versus \\a small path or cycle}
\author{Sha Wang$^{1,2}$,\, Ruyu Song$^{1,2}$,\, Yixin Zhang$^{1,2}$,\, Yanbo Zhang$^{1,2,}$\thanks{Corresponding author: {\tt ybzhang@hebtu.edu.cn}. Research supported by NSFC (No.\ 11601527 and 11971011).}}

 \affil{ { \small {$^1$School of Mathematical Sciences, Hebei Normal University, Shijiazhuang 050024, P.R.~China}}\\
	     { \small {$^2$Hebei International Joint Research Center for Mathematics and Interdisciplinary Science,\\ Shijiazhuang 050024, P.R.~China}}
}

\date{}

\begin{document}
\maketitle
\begin{abstract}
	Given two graphs $G_1, G_2$, the connected size Ramsey number $\rc(G_1,G_2)$ is defined to be the minimum number of edges of a connected graph $G$, such that for any red-blue edge colouring of $G$, there is either a red copy of $G_1$ or a blue copy of $G_2$. Concentrating on $\rc(nK_2,G_2)$ where $nK_2$ is a matching, we generalise and improve two previous results as follows. Vito, Nabila, Safitri, and Silaban obtained the exact values of $\rc(nK_2,P_3)$ for $n=2,3,4$. We determine its exact values for all positive integers $n$. Rahadjeng, Baskoro, and Assiyatun proved that $\rc(nK_2,C_4)\le 5n-1$ for $n\ge 4$. We improve the upper bound from $5n-1$ to $\lfloor (9n-1)/2 \rfloor$. In addition, we show a result which has the same flavour and has exact values: $\rc(nK_2,C_3)=4n-1$ for all positive integers $n$.
\end{abstract}

\section{Introduction}
Graph Ramsey theory is currently among the most active areas in combinatorics. Two of the main parameters in the theory are Ramsey number and size Ramsey number, which are defined as follows. Given two graphs $G_1$ and $G_2$, we write $G\to (G_1,G_2)$ if for any edge colouring of $G$ such that each edge is coloured either red or blue, the graph $G$ always contains either a red copy of $G_1$ or a blue copy of $G_2$. The \emph{Ramsey number} $r(G_1,G_2)$ is the smallest possible number of vertices in a graph $G$ satisfying $G\to (G_1,G_2)$. The \emph{size Ramsey number} $\hat{r}(G_1,G_2)$ is the smallest possible number of edges in a graph $G$ satisfying $G\to (G_1,G_2)$. That is to say, $r(G_1,G_2)=\min\{|V(G)|:G\to (G_1,G_2)\}$, and $\hat{r}(G_1,G_2)=\min\{|E(G)|:G\to (G_1,G_2)\}$.

The size Ramsey number was introduced by Erd\H os, Faudree, Rousseau, and Schelp \cite{erdos1978size} in 1978. Some variants have also been studied since then. In 2015, Rahadjeng, Baskoro, and Assiyatun \cite{rahadjeng2015connected} initiated the study of such a variant called connected size Ramsey number by adding the condition that $G$ is connected. Formally speaking, the \emph{connected size Ramsey number} $\rc(G_1,G_2)$ is the smallest possible number of edges in a connected graph $G$ satisfying $G\to (G_1,G_2)$. It is easy to see that $\hat{r}(G_1,G_2)\le \rc(G_1,G_2)$, and equality holds when both $G_1$ and $G_2$ are connected graphs. But the latter parameter seems more tricky when $G_1$ or $G_2$ is disconnected. The previous results are mainly concerned with the connected size Ramsey numbers of a matching versus a sparse graph such as a path, a star, and a cycle.

Let $nK_2$ be a matching with $n$ edges, and $P_m$ a path with $m$ vertices. Vito, Nabila, Safitri, and Silaban \cite{vito2021size} gave an upper bound of $\rc(nK_2,P_m)$, and the exact values of $\rc(nK_2,P_3)$ for $n=2,3,4$.

\begin{thm}\cite{vito2021size}
	\label{thm:pm}
	For $n\ge 1$, $m\ge 3$, $\rc(nK_2,P_m)\le \begin{cases}n(m+2)/2-1, & \text{if}\ n\ \text{is even}; \\ (n+1)(m+2)/2-3, & \text{if}\ n\ \text{is odd}. \end{cases}$\\ Equality holds for $m=3$ and $1\le n\le 4$.
\end{thm}

If $m$ is much larger than $n$, this upper bound cannot be tight. Because Erd\H os and Faudree \cite{erdos1984size} constructed a connected graph which implies $\rc(nK_2,P_m)\le m+c\sqrt{m}$, where $c$ is a constant depending on $n$. But for small $m$, the above upper bound can be tight. Our first result determines the exact values of $\rc(nK_2,P_3)$ for all positive integers $n$, which generalises the equality of \cref{thm:pm}.
\begin{thm}
	\label{thm:P3}
	For all positive integers $n$, we have $\rc(nK_2,P_3)=\flo{(5n-1)/2}$.
\end{thm}

Rahadjeng, Baskoro, and Assiyatun \cite{rahadjeng2017connected} proved that $\rc(nK_2,C_4)\le 5n-1$ for $n\ge 4$. This upper bound can be improved from $5n-1$ to $\lfloor (9n-1)/2 \rfloor$. 

\begin{thm}
	\label{thm:C4}
	For all positive integers $n$, we have $\rc(nK_2,C_4)\le \lfloor (9n-1)/2 \rfloor$.
\end{thm}

Now we prove the theorem by constructing a graph with $\lfloor (9n-1)/2 \rfloor$ edges. Let $K_{3,3}-e$ be the graph $K_{3,3}$ with one edge deleted. It is easy to check that $K_{3,3}-e\to (2K_2,C_4)$. We use $nG$ to denote $n$ disjoint copies of $G$. If $n$ is even, then $\frac{n}{2}(K_{3,3}-e)\to (nK_2,C_4)$. The graph $\frac{n}{2}(K_{3,3}-e)$ has $n/2$ components and can be connected by adding $n/2-1$ edges. If $n$ is odd, then $\frac{n-1}{2}(K_{3,3}-e)\cup C_4\to (nK_2,C_4)$. The graph $\frac{n-1}{2}(K_{3,3}-e)\cup C_4$ has $(n+1)/2$ components and can be connected by adding $(n-1)/2$ edges. In both cases, we obtain a connected graph with $\lfloor (9n-1)/2 \rfloor$ edges and hence the upper bound follows.

It seems likely that the determination of $\rc(nK_2,C_4)$ for all $n$ is tricky. We believe the upper bound is tight and pose the following conjecture.

\begin{conj}
	\label{conj:C4}
	For all positive integers $n$, $\rc(nK_2,C_4)=\lfloor (9n-1)/2 \rfloor$.
\end{conj}

Even though solving the above conjecture seems out of our reach, we show a result which has the same flavour and has exact values: $\rc(nK_2,C_3)=4n-1$.

\begin{thm}
	\label{thm:C3}
	For all positive integers $n$, we have $\rc(nK_2,C_3)=4n-1$.
\end{thm}

Proofs of \cref{thm:P3} and \cref{thm:C3} will be presented in Section \ref{section2} and Section \ref{section3}, respectively. To prove the lower bounds, we need to discuss the connectivity of a graph $G$. If $G$ is not 2-connected, the basic properties of blocks and end blocks are needed, which can be found in Bondy and Murty \cite[Chap.~5.2]{bondy2008graph}. Moreover, the following terminology is used frequently in the proofs. We say $G$ has a \emph{$(G_1,G_2)$-colouring} if there is a red-blue edge colouring of $G$ such that $G$ contains neither a red $G_1$ nor a blue $G_2$. Thus, it is equivalent to $G\not\to (G_1,G_2)$.

\section{A matching versus $P_3$}\label{section2}
	For the upper bound, we know that $C_4\to (2K_2,P_3)$. If $n$ is even, then $\frac{n}{2}C_4\to (nK_2,P_3)$. The graph $\frac{n}{2}C_4$ has $n/2$ components and can be connected by adding $n/2-1$ edges. If $n$ is odd, then $\frac{n-1}{2}C_4\cup P_3\to (nK_2,P_3)$. The graph $\frac{n-1}{2}C_4\cup P_3$ has $(n+1)/2$ components and can be connected by adding $(n-1)/2$ edges. In both cases, we obtain a connected graph with $\flo{(5n-1)/2}$ edges and hence the upper bound follows.
	
	For the lower bound, we use induction on $n$. The result is obvious for $n=1,2$. Assume that for $k<n$ and any connected graph $G$ with $\flo{(5k-3)/2}$ edges, we have $G\not\to (kK_2,P_3)$. Now consider $G$ to be a connected graph with minimum number of edges such that $G\to (nK_2,P_3)$. Thus, for any proper connected subgraph $G'$ of $G$, we have $G'\not\to (nK_2,P_3)$. Since $n\ge 3$, $G$ has at least six edges. Suppose to the contrary that $G$ has at most $\flo{(5n-3)/2}$ edges. We will deduce a contradiction and hence $\rc(nK_2,P_3)\ge \flo{(5n-1)/2}$.

	An edge set $E_1$ of a connected graph $G$ is called \emph{deletable}, if $E_1$ satisfies the following conditions:
	\begin{lemenum}
		\item $E_1$ can be partitioned into two edge sets $E_2$ and $E_3$, where $E_2$ forms a star and $E_3$ forms a matching;
		\item any edge of $E(G)\setminus E_1$ is nonadjacent to $E_3$;
		\item the graph induced by $E(G)\setminus E_1$ is still connected.
	\end{lemenum}
	Note that for a deletable edge set $E_1$, the graph $G-E_1$ may have some isolated vertices, but all edges of $G-E_1$ belong to the same connected component. We have the following property of a deletable edge set. 
	
	\begin{claim}\label{clm: deletableedge}
		Every deletable edge set has size at most two.
	\end{claim}
	
	\begin{proof}
		Let $E_1$ be a deletable edge set. If $|E_1|\ge 3$, then the graph induced by $E(G)\setminus E_1$ has at most $\flo{(5n-3)/2}-3$ edges and hence an $((n-1)K_2, P_3)$-colouring by induction. We then colour all edges of $E_2$ red and all edges of $E_3$ blue. This is a $(nK_2, P_3)$-colouring of $G$, a contradiction.
	\end{proof}
	
	A \emph{non-cut vertex} of a connected graph is a vertex whose deletion still results in a connected graph. That is, every vertex of a nontrivial connected graph is either a cut vertex or a non-cut vertex. Since $E_3$ in the definition of a deletable edge set can be empty, the edges incident to a non-cut vertex form a deletable edge set. We have the following direct corollary. 
	
	\begin{claim}\label{clm: noncut}
		Every non-cut vertex has degree at most two.
	\end{claim}	
	
	If $G$ is a 2-connected graph, by \cref{clm: noncut}, $G$ is a cycle. Beginning from any edge of $G$, we may colour all edges consecutively along the cycle. We alternately colour two edges red and one edge blue, until all edges of $G$ have been coloured. Obviously $G$ contains no blue $P_3$. From $(5n-3)/2\le 3(n-1)$ and the colouring of $G$ we see that $G$ contains no red matching with $n$ edges. Thus, $G\not\to (nK_2,P_3)$.
	
	Now we assume that $G$ is connected but not 2-connected. Recall that a \emph{block} of a graph is a subgraph that is nonseparable and is maximal with respect to this property. An \emph{end block} is a block that contains exactly one cut vertex of $G$. We have the following observation.
	
	\begin{claim}\label{clm: endblock}
		Every end block is either a $K_2$ or a cycle.
	\end{claim}
	
	\begin{proof}
		Let $B$ be an end block with at least three vertices, and let $v$ be the single cut vertex of $G$ that is contained in $B$. Since $B$ is 2-connected, the subgraph $B-v$ is still connected. By \cref{clm: noncut}, every non-cut vertex has degree two. It follows that $B-v$ is either a path or a cycle. We see that $v$ has two neighbours in $B$. If not, $v$ has at least three neighbours in $B$, each of which has degree one in $B-v$. Since a path has two vertices of degree one and a cycle has no vertex of degree one, $B-v$ is neither a path nor a cycle, a contradiction. Hence, every vertex of $B$ has two neighbours in $B$. Since $B$ is 2-connected, it must be a cycle.
	\end{proof}
	
	Since $G$ is not 2-connected, there is at least one cut vertex. Choose any cut vertex as a \emph{root}, denoted by $r$. For a vertex $u$ of $G$, if any path from $u$ to $r$ must pass through a cut vertex $v$, then $u$ is called a \emph{(vertex) descendant} of $v$. For any edge $e$ of $G$, if both ends of $e$ are descendants of $v$, then $e$ is called an \emph{edge descendant} of $v$. For a cut vertex $v$, the block containing $v$ but no other descendant of $v$ is called a \emph{parent block} of $v$. It is obvious that every cut vertex has a unique parent block, except that the root $r$ has no parent block. If $v$ is a cut vertex but every descendant of $v$ is not a cut vertex of $G$, we call $v$ an \emph{end-cut}. It is obvious that $G$ has at least one end-cut. We have the following property of end-cuts.
	
	\begin{claim}\label{clm: end-cut}
		Every end-cut is contained in a unique end block, which is $K_2$. Moreover, if an end-cut is not the root of $G$, its parent block is also $K_2$.
	\end{claim}
	
	\begin{proof}
		Let $v$ be an end-cut. If $v$ is not the root $r$, by the definition of end-cut, every block containing $v$ is an end block, except for its parent block. If we delete $v$ and all descendants of $v$ from $G$, the induced subgraph is still  connected, denoted by $G'$. This is because, any vertex of $G'$ is not a descendant of $v$. For any two vertices of $G'$, there is a path joining them in $G$ without passing through $v$. So the path still exists in $G'$ and hence $G'$ is connected. In the following, regardless of whether $v$ is the root or not, we first colour all edges incident to $v$ red, then give a colouring of all edge descendants of $v$. After that, we find a colouring of $G'$ by the inductive hypothesis. We prove that this edge colouring of $G$ is an $(nK_2,P_3)$-colouring under certain conditions.
		
		By \cref{clm: endblock}, every end block is either a $K_2$ or a cycle. Assume that $v$ has $t_1$ neighbours in its parent block. Note that if $v$ is the root of $G$, then $t_1=0$. Assume $v$ is contained in $t_2$ blocks which are $K_2$, and in $t$ blocks which are cycles. Let $p_1+2, p_2+2, \dots, p_t+2$ be the cycle lengths of these $t$ cycles. If we remove $v$ from $G$, the cycles become $t$ disjoint paths with length $p_1, p_2, \dots, p_t$ respectively. We colour all edges incident to $v$ red. For each path with length $p_i$, where $1\le i\le t$, we colour all edges from one leaf to the other leaf consecutively along the path, alternately with one edge blue and two edges red. Now we have coloured $x:=t_1+t_2+2t+p_1+p_2+\dots+p_t$ edges, no blue $P_3$ appears, and the maximum red matching has $y:=1+\flo{(p_1+1)/3}+\dots+\flo{(p_t+1)/3}$ edges.
		
		If $v$ is the root, then we have already coloured all edges of $G$. So $x\le \flo{(5n-3)/2}$, and we need to check that $y\le n-1$. Since $G$ has at least six edges, we have $6t_2+7t+p_1\ge 11$. Thus, $n\ge (2x+3)/5\ge (2t_2+4t+2p_1+\dots+2p_t+3)/5\ge (t+p_1+\dots+p_t+4)/3>y$. This implies that $G\not\to (nK_2,P_3)$.
		
		If $v$ is not the root, recall that $G'$ is formed by the remaining edges and is connected. We can use the inductive hypothesis. The graph $G'$ has at most $\flo{(5n-3)/2}-x$ edges, which is $\flo{(5(n-2x/5)-3)/2}\le \flo{(5(n-\flo{2x/5})-3)/2}$. So $G'$ has an $((n-\flo{2x/5})K_2,P_3)$-colouring. It is not difficult to check that $G$ has no blue $P_3$, and the maximum red matching has at most $y+n-\flo{2x/5}-1$ edges. It is left to show that under what conditions $y+n-\flo{2x/5}-1$ is less than $n$. So we deduce a contradiction.
		
		Since $v$ has at least one neighbour in its parent block, it follows that $t_1\ge 1$. If $t\ge 1$, then $1+(t-1)/3\le (4t+1)/5$, and $\flo{(p_1+1)/3}\le (2p_1+1)/5$. Thus,
		\begin{align*}			
			y&=1+\flo{(p_1+1)/3}+\flo{(p_2+1)/3}+\dots+\flo{(p_t+1)/3}\\
			&\le 1+\flo{(p_1+1)/3}+(p_2+1)/3+\dots+(p_t+1)/3\\
			&\le 1+\flo{(p_1+1)/3}+(t-1)/3+2(p_2+\dots+p_n)/5\\
			&\le (4t+1)/5+(2p_1+1)/5+2(p_2+\dots+p_n)/5\\
			&\le 2(t_1+t_2+2t+p_1+p_2+\dots+p_t)/5=2x/5. 	
		\end{align*}
	If $t=0$ and $t_1+t_2\ge 3$, then $y=1<6/5\le 2x/5$. In both cases, we have $y\le 2x/5$ and hence $y+n-\flo{2x/5}-1<n$.
		
		Now we consider the remainder case when $t=0$ and $t_1+t_2\le 2$. Since $v$ is an end-cut, we have $t_2\ge 1$. Thus, $t_1=t_2=1$. It follows from $t=0$ and $t_2=1$ that $v$ is contained in a unique end block which is $K_2$. It follows from $t_1=1$ that $v$ has only one neighbour in its parent block. Since each block is either a $K_2$ or a 2-connected subgraph, the parent block of $v$ must be $K_2$.
	\end{proof}
	
	Let $v$ be an end-cut. By \cref{clm: end-cut}, it cannot be the root of $G$. Let $u$ be the other end of its parent block, and $v^+$ the descendant of $v$. If $u$ is contained in an end block which is an edge $uu^+$, then $uv, uu^+, vv^+$ form a deletable edge set. If $u$ is contained in an end block which is not an edge, by \cref{clm: endblock}, the end block is a cycle. Let $u^+$ be a neighbour of $u$ on the cycle. Then $uv, uu^+, vv^+$ form a deletable edge set. If $u$ has at least two end-cuts as its descendants, let $w$ be another end-cut and $uw, ww^+$ edge descendants of $u$. Then $uv, vv^+, uw, ww^+$ form a deletable edge set. If $u$ has only one end-cut as its descendant, which is $v$, then all edges incident to $u$ and the edge $vv^+$ form a deletable edge set. By \cref{clm: deletableedge}, each of the above cases leads to a contradiction. This completes the proof of the lower bound.

\section{A matching versus $C_3$}\label{section3}
Now we prove \cref{thm:C3}. The graph $nC_3$ has $n$ components and can be connected by adding $n-1$ edges. Denote this connected graph by $H$. It follows from $nC_3\to (nK_2,C_3)$ that $H\to (nK_2,C_3)$. Thus, $\rc(nK_2,C_3)\le 4n-1$. 

Set $\mathcal{G}=\{G:|E(G)|\le 4n-2, \text{and}\ G\to (nK_2,C_3)\}$. We will prove that $\mathcal{G}$ is an empty set and hence the lower bound follows. Suppose not, choose a graph $G$ from $\mathcal{G}$ with minimum order and minimum size subjective to its order. Thus, for any proper connected subgraph $G'$ of $G$ with at most $4k-2$ edges, we have $G'\not\to (kK_2,C_3)$. We present the proof through a sequence of claims.

	\begin{claim}\label{clm:minimumdegree}
	The minimum degree of $G$ is at least two.
    \end{claim}

    \begin{proof}
	Suppose that $G$ has a vertex $v$ of degree one. Then $G-v$ has an $(nK_2,C_3)$-colouring. It can be extended to an $(nK_2,C_3)$-colouring of $G$ by colouring the edge incident to $v$ blue, which contradicts our assumption that $G\to (nK_2,C_3)$. Thus, $\delta(G)\ge 2$.
    \end{proof}

	\begin{claim}\label{clm:nocutedge}
	The graph $G$ has no cut edge.
    \end{claim}

    \begin{proof}
	Suppose that $G$ has a cut edge $e$. Then $G-e$ has two connected components $X$ and $Y$. Let $k_1,k_2$ be the integers such that $4k_1-5\le |E(X)|\le 4k_1-2$ and $4k_2-5\le |E(Y)|\le 4k_2-2$ respectively. Then $X$ has a $(k_1K_2,C_3)$-colouring and $Y$ has a $(k_2K_2,C_3)$-colouring. It can be extended to an $((k_1+k_2-1)K_2,C_3)$-colouring of $G$ by colouring $e$ blue. So the maximum red matching has at most $k_1+k_2-2$ edges. From $(4k_1-5)+(4k_2-5)+1\le |E(X)|+|E(Y)|+1\le 4n-2$ we deduce that $k_1+k_2-2\le n-1/4$, a contradiction which implies our claim.
   \end{proof}

    \begin{claim}\label{clm:2-connected}
	The graph $G$ is 2-connected.
    \end{claim}

    \begin{proof}
	If $G$ is not 2-connected, let $v$ be a cut vertex of $G$, and $B_1,B_2,\dots,B_\ell$ the blocks containing $v$, where $\ell\ge 2$. If $v$ has only one neighbour in $B_i$ for some $i$ with $1\le i\le \ell$, then $B_i$ is not 2-connected. Since any block is either 2-connected or a $K_2$, $B_i$ has to be $K_2$. Hence, $B_i$ is a cut edge \cite[Chap.~4.1.18]{west2001introduction}, which contradicts \cref{clm:nocutedge}. Thus, for each $B_i$ with $1\le i\le \ell$, $v$ has at least two neighbours in $B_i$.
	
	The vertex set $V(G)$ can be partitioned into two parts $X$ and $Y$ as follows. If any path from $u$ to $v$ has to pass through a vertex of $B_1$ other than $v$, then $u\in X$; otherwise $u\in Y$. Let $G_1$ and $G_2$ be the subgraphs induced by $X\cup \{v\}$ and $Y$ respectively. It is obvious that $G_1$ contains $B_1$ and $G_2$ contains $B_2\cup \dots\cup B_\ell$. And they share only one vertex, which is $v$. Let $k_1,k_2$ be the integers such that $4k_1-5\le |E(G_1)|\le 4k_1-2$ and $4k_2-5\le |E(G_2)|\le 4k_2-2$ respectively. Then $G_1$ has a $(k_1K_2,C_3)$-colouring and $G_2$ has a $(k_2K_2,C_3)$-colouring. Combining the two colourings we have an $((k_1+k_2-1)K_2,C_3)$-colouring of $G$. So the maximum red matching has at most $k_1+k_2-2$ edges. From $(4k_1-5)+(4k_2-5)\le |E(G_1)|+|E(G_2)|\le 4n-2$ we deduce that $k_1+k_2-2\le n$. If $k_1+k_2-2<n$, then this colouring is an $(nK_2,C_3)$-colouring of $G$. If $k_1+k_2-2=n$, then we have $|E(G_1)|=4k_1-5$ and $|E(G_2)|=4k_2-5$. We obtain an $(nK_2,C_3)$-colouring of $G$ as follows. If $\ell=2$, then both $B_1-v$ and $B_2-v$ are connected. Hence, both $G_1-v$ and $G_2-v$ are connected, which implies that they have a $((k_1-1)K_2,C_3)$-colouring and a $((k_2-1)K_2,C_3)$-colouring, respectively. It can be extended to an $((k_1+k_2-2)K_2,C_3)$-colouring of $G$ by colouring all edges incident to $v$ red. Thus, $G\not\to (nK_2,C_3)$. If $\ell\ge 3$, then for each $i$ with $2\le i\le \ell$, we delete an edge $vv_i$ from $B_i$. Both $G_1-v$ and $G_2-\{vv_2,\dots,vv_\ell\}$ are connected. So they have a $((k_1-1)K_2,C_3)$-colouring and a $((k_2-1)K_2,C_3)$-colouring, respectively. It can be extended to an $((k_1+k_2-2)K_2,C_3)$-colouring of $G$ by colouring the remaining edges red. Again, $G\not\to (nK_2,C_3)$.
    \end{proof}

    \begin{claim}\label{clm:maximumdegree}
	The maximum degree of $G$ is at most three.
    \end{claim}

    \begin{proof}
	For any vertex $v$ of $G$, by \cref{clm:2-connected}, $G-v$ is still connected. If $d(v)\ge 4$, $G-v$ has at most $4(n-1)-2$ edges and hence an $((n-1)K_2,C_3)$-colouring by the choice of $G$. It can be extended to an $(nK_2,C_3)$-colouring of $G$ by colouring all edges incident to $v$ red, a contradiction. Thus, the maximum degree of $G$ is at most three.
    \end{proof}

    \begin{claim}\label{clm:3-regular}
	The graph $G$ is 3-regular.
\end{claim}

\begin{proof}
	By \cref{clm:minimumdegree} and \cref{clm:maximumdegree}, $2\le d(v)\le 3$ for any vertex $v$ of $G$. If $G$ is 2-regular, by \cref{clm:2-connected}, $G$ is a cycle. If $G$ is a triangle, then $n\ge 2$. We colour all edges of $G$ red, which is a $(nK_2,C_3)$-colouring. If $G$ is not a triangle, then we colour all edges of $G$ blue, which is a $(K_2,C_3)$-colouring. Thus, $G$ cannot be 2-regular.
	
	If $G$ is not 3-regular, then $G$ has some vertices with degree two and some with degree three. There exist two adjacent vertices with degrees two and three respectively, denoted by $v_1$ and $v_2$. Since $G-v_2$ is connected, and $v_1$ has only one neighbour in $G-v_2$, it follows that $G-\{v_1,v_2\}$ is connected. This graph has at most $4(n-1)-2$ edges and hence an $((n-1)K_2,C_3)$-colouring. It can be extended to an $(nK_2,C_3)$-colouring of $G$ by colouring all edges incident to $v_2$ red and the remaining edge incident to $v_1$ blue, a contradiction which implies our claim. 
    \end{proof}

	\begin{claim}\label{clm:triangle}
	Each edge of $G$ is contained in at least one triangle.
    \end{claim}

    \begin{proof}
	Suppose that $G$ has an edge $e$ which is not contained in any triangle. By \cref{clm:nocutedge}, $G-e$ is connected. It follows from the choice of $G$ that $G-e$ has an $(nK_2,C_3)$-colouring. It can be extended to an $(nK_2,C_3)$-colouring of $G$ by colouring $e$ blue, a contradiction.
    \end{proof}

Consider a triangle $v_1v_2v_3$. By \cref{clm:3-regular}, each of $v_1,v_2,v_3$ has another neighbour, denoted by $v_4,v_5,v_6$ respectively. If $v_4,v_5,v_6$ are the same vertex, then $v_1,v_2,v_3,v_4$ forms a $K_4$. Since $G$ is a 3-regular 2-connected graph, the whole graph $G$ is a $K_4$ and $n\ge 2$. We colour the triangle $v_1v_2v_3$ red, and the other three edges blue. This is a $(2K_2,C_3)$-colouring of $G$, a contradiction. Thus, at least two of $v_4,v_5,v_6$ are distinct, say, $v_4$ and $v_5$ are two distinct vertices. The vertex $v_3$ cannot be adjacent to both $v_4$ and $v_5$, since otherwise $d(v_3)\ge 4$. Without loss of generality, assume that $v_3$ is not adjacent to $v_4$. Moreover, $v_4$ is not adjacent to $v_2$, since otherwise $d(v_2)\ge 4$. By \cref{clm:triangle}, $v_1v_4$ is contained in a triangle, denoted by $v_1v_4v_6$. Since $v_6$ is different with $v_2,v_3$, we have $d(v_1)\ge 4$, a final contradiction.


\begin{thebibliography}{1}
	
	\bibitem{bondy2008graph}
	J.~A. Bondy and U.~S.~R. Murty, \emph{Graph theory}, Springer, 2008.
	
	\bibitem{erdos1984size}
	P.~Erd{\H{o}}s and R.~Faudree, \emph{Size Ramsey numbers involving matchings},
	Finite and Infinite Sets, Elsevier, 1984, pp.~247--264.
	
	\bibitem{erdos1978size}
	P.~Erd{\H{o}}s, R.~Faudree, C.~C. Rousseau, and R.~H. Schelp, \emph{The size
		Ramsey number}, Period. Math. Hungar. \textbf{9} (1978), 145--161.
	
	\bibitem{rahadjeng2015connected}
	B.~Rahadjeng, E.~T. Baskoro, and H.~Assiyatun, \emph{Connected size Ramsey
		numbers for matchings versus cycles or paths}, Procedia Comput. Sci.
	\textbf{74} (2015), 32--37.
	
	\bibitem{rahadjeng2017connected}
	B.~Rahadjeng, E.~T. Baskoro, and H.~Assiyatun, \emph{Connected size Ramsey
		number for matchings vs. small stars or cycles}, Proc. Indian Acad. Sci.:
	Math. Sci. \textbf{127} (2017), 787--792.
	
	\bibitem{vito2021size}
	V.~Vito, A.~C. Nabila, E.~Safitri, and D.~R. Silaban, \emph{The size Ramsey and
		connected size Ramsey numbers for matchings versus paths}, J. Phys. Conf.
	Ser. \textbf{1725} (2021), p.~012098.
	
	\bibitem{west2001introduction}
	D.~B. West, \emph{Introduction to graph theory}, vol.~2, Prentice Hall, Upper
	Saddle River, 2001.
	
\end{thebibliography}
\end{document}